\documentclass[12pt,leqno,oneside]{amsart}
\usepackage{mathrsfs,dsfont}
\usepackage{amsmath, amstext, amsthm, amssymb, bbm, comment}
\usepackage{charter}
\usepackage{typearea}
\usepackage{hyperref}
\usepackage{color}

\usepackage{float}
\usepackage{tikz}
\usetikzlibrary{patterns}
\usepackage{graphicx}
\usepackage{hyperref}

\usepackage[width=6.4in,height=8.5in]
{geometry}

\def\bt{\begin{thm}}
	\def\et{\end{thm}}
\def\bl{\begin{lem}}
	\def\el{\end{lem}}
\def\bd{\begin{defn}}
	\def\ed{\end{defn}}
\def\bc{\begin{cor}}
	\def\ec{\end{cor}}
\def\bp{\begin{proof}}
	\def\ep{\end{proof}}
\def\br{\begin{rem}}
	\def\er{\end{rem}}

\newtheorem{thm}{Theorem}[section]
\newtheorem{prop}[thm]{Proposition}
\newtheorem{lem}[thm]{Lemma}
\newtheorem{defn}[thm]{Definition}
\newtheorem{example}[thm]{Example}
\newtheorem{rem}[thm]{Remark}
\newtheorem{cor}[thm]{Corollary}

\numberwithin{equation}{section}

\title{On Dynamics of Asymptotically Minimal Polynomials}

\author{Turgay Bayraktar} 
\author{Mel\.{I}ke Efe} 
\thanks{T.\ Bayraktar is partially supported by Turkish Academy of Sciences, GEBIP grant.}
\thanks{M. Efe is supported by T\"{U}B\.{I}TAK grant ARDEB-3501/119F184}
\address{Faculty of Engineering and Natural Sciences, Sabanc{\i} University, \.{I}stanbul, Turkey}
\email{tbayraktar@sabanciuniv.edu}
\email{melikeefe@sabanciuniv.edu}


\date{\today}

\begin{document}
	
	
	\maketitle
	\begin{abstract}
		We study dynamical properties of asymptotically extremal polynomials associated with a non-polar planar compact set $E$. In particular, we prove that if the zeros of such polynomials are uniformly bounded then their Brolin measures converge weakly to the equilibrium measure of $E$. In addition, if $E$ is regular and the zeros of such polynomials are sufficiently close to $E$ then we show that the filled Julia sets converge to polynomial convex hull of $E$ in the Klimek topology.  
\end{abstract}
	
	\section{Introduction}
Let $E\subset \Bbb{C}$ be a compact set with positive logarithmic capacity $\mathrm{cap}(E)$. We denote by $\Omega$ the unbounded component of  $\widehat{\mathbb{C}}\setminus E$ so that $\mathrm{Pc}(E):=\Bbb{C}\setminus \Omega$ is the polynomial convex hull of $E$. We also let $\omega_E$ be the equilibrium measure of $E$ that is the unique maximizer of the logarithmic energy among all Borel probability measures supported on $E$ (see \cite{ransfordbook}).  

Recall that a Borel measure belongs to the class \textbf{Reg} if the $n^{th}$ root of the leading coefficient of the $n^{th}$ orthonormal polynomial is asymptotic to logarithmic capacity of the support of the measure as the degree $n$ grows to infinity (see \S 2.1.1 for precise definition). For example, equilibrium measure of a regular compact set is of class \textbf{Reg}. In this paper, we focus on dynamical properties of asymptotically extremal polynomials associated with planer compact sets and regular measures. More precisely,  following \cite{duncan} (cf. \cite[Chp 3]{general orth. polyn.}) we define:

\begin{defn}\label{definition asyp. min.}     
We say that a sequence of polynomials $p_n(z)=\sum_{j=0}^{n}a_{n,j}z^{j}$ with $a_{n,n}\not=0$ is \textit{asymptotically minimal} on a compact set $ E\subset\mathbb{C} $ if there exists a regular measure $ \tau\in\textbf{Reg}$ with $ supp(\tau)=E $ and a constant $ p\in (0,\infty] $ such that 
    \begin{equation}\label{definition asymp. min. eq. 1}
			\lim_{n\rightarrow\infty}\dfrac{1}{n}\log|a_{n,n}|=-\log \mathrm{cap}(E)
		\end{equation} and 
		\begin{equation}\label{definition asymp. min. eq. 2}
			\lim_{n\rightarrow\infty}\dfrac{1}{n}\log\parallel p_n \parallel_{\mathnormal{L}^{p}(\tau)}=0.
		\end{equation}
	\end{defn}
Note that for the case $p=\infty$ we do not need the reference measure $\tau$ in the definition. We remark that by Proposition \ref{ASMp} if the conditions (\ref{definition asymp. min. eq. 1}) and (\ref{definition asymp. min. eq. 2}) hold for one $p\in (0,\infty)$ then they hold for every $p\in(0,\infty]$. Orthonormal polynomials, normalized Chebyshev, Fekete and Faber polynomials are primary examples of asymptotically minimal  polynomials. In addition, polynomials with uniformly bounded coefficients are also asymptotically minimal on the unit circle with respect to the arc length measure (see \S \ref{examples} for more details). We also remark that Definition \ref{definition asyp. min.} is motivated by the fact that the ratio $\frac1n\log(\|p_n\|_{L^p(\tau)}/|a_{n,n}|)$ converges to the minimal value $\log \mathrm{cap}(E)$ (see \cite[\S 3]{general orth. polyn.}). However, merely this last condition is not sufficient for our purposes (cf. Lemma \ref{greens functions} and Example \ref{Ex}). Thus, we require both the convergence of $n^{th}$ roots of the leading coefficients (\ref{definition asymp. min. eq. 1}) and the corresponding norms (\ref{definition asymp. min. eq. 2}). Asymptotic distribution of normalized counting zero measures of extremal polynomials on a planar compact set are closely related to the equilibrium measure of the set and this relation  has been studied extensively in the literature (see \cite{general orth. polyn., SaTo} references therein). 
	
Recall that for a non-linear polynomial $p(z)$, the filled Julia set $K_p$ is defined to be the set of all points with bounded forward orbit under $p$. The set $K_p$ is a regular polynomially convex compact set in $\Bbb{C}$. The topological boundary $J_p:=\partial K_p$ is called the Julia set of $p$. It follows from \cite{brolin, lyubich} that the equilibrium measure $\omega_{J_p}$ is the unique $p$-invariant measure of maximal entropy (see \cite{CG,Sibony} and references therein for a detailed account of the subject).  

In this note, we study the relation between the potential theoretic equilibrium measure $\omega_E$ (respectively its support $\partial\Omega$) and the dynamically defined Brolin measures (respectively the Julia sets) of asymptotically minimal polynomials $\{p_n\}_{n\in\Bbb{N}}$ on a planar compact set $E$. In this context, Chiristiansen et. al. studied the limit behavior of the filled Julia sets $K_{p_n}$ of orthogonal polynomials with respect to the Hausdorff topology \cite{juliasetortpoly}. Later, in \cite{weaklmtof maximal entropy}  Petersen and Uhre showed that Brolin measures of orthogonal polynomials associated with a \textbf{Reg} class measure converge in weak* topology to the equilibrium measure of the set. More recently,  this result was generalized to the case of normalized Chebyshev polynomials in \cite{filled julia set of cheb. polyn}. In this note, by adapting the techniques of \cite{weaklmtof maximal entropy, LY, filled julia set of cheb. polyn} to our setting we obtain another generalization for asymptotically minimal polynomials:
	\begin{thm}\label{mainteorem1}
		Let $E\subset \Bbb{C}$ be a compact set with positive logarithmic capacity and $\{p_n\}_{n\in \Bbb{N}}$ be a sequence of asymptotically minimal polynomials on $E$. Assume that the zeros of $p_n(z)$ are uniformly bounded for $n\in\Bbb{N}$. Then, the Brolin measures $\omega_{J_{p_n}}\to \omega_E $ in the weak* topology.
	\end{thm}  
Another aspect of research in this context is approximating a given planar compact set by polynomial filled Julia sets (respectively Julia sets) with respect to the Hausdorff topology (cf. \cite{L, BP, LY, BKS}). In this work, we consider a different problem: Namely, classifying all possible geometric limit sets of a sequence of filled Julia sets of asymptotically minimal polynomials.  First, we observe that the sequence of filled Julia sets (respectively Julia sets) of asymptotically minimal polynomials may not  converge in the Hausdorff topology (see Example \ref{dynamicalexample}). Moreover, under the stronger assumption on zeros of asymptotically minimal polynomials associated with a regular compact set we show that the filled Julia sets of the extremal polynomials converge with respect to a natural metric $\Gamma$, called the Klimek metric \cite{Klimek}, that is defined in terms of Greens functions of the corresponding regular polynomially convex compact sets in $\Bbb{C}$ (see (\ref{Gamma}) for the definition). More precisely, we prove the following result:	
				
\begin{thm}\label{approximation by filled julia sets}
Let  $ E $ be a regular compact set and $ \{p_{n}\}_{n\in \Bbb{N}} $ be a sequence of  asymptotically minimal polynomials on $E$. Assume that for each $ \varepsilon>0 $ there exists $ N\in\mathbb{N} $ such that for all $ n\geq N $ the zeros of $p_n $ are contained in $\varepsilon$-dilation of $ \mathrm{Pc}(E) $. Then $\Gamma(K_{p_n}, \mathrm{Pc}(E))\to 0$ as $n\to\infty$. 
\end{thm}	

Finally, as a corollary of Theorem \ref{approximation by filled julia sets}, we also prove that for any Hausdorff-limit set $K_{\infty}$ of filled Julia sets $\{K_n\}_{n\in\Bbb{N}}$ we have $\mathrm{Pc}(K_{\infty})=\mathrm{Pc}(E)$ (Proposition \ref{limitsets}). 
	 			
		\section{Preliminaries}
		\subsection{Potential Theory}
		
		Throughout this section, we shall use \cite{ransfordbook} as a reference to the basic concepts of potential theory in the complex plane. Let $ E $ be a compact subset of $ \mathbb{C} $ and $ \mu $ be a finite Borel measure with support $ E $. The potential of $ \mu $ is a function $ U_{\mu}:\mathbb{C}\rightarrow[-\infty,\infty) $ defined by \begin{equation}
			U_{\mu}(z)=\int \log|z-w|d\mu(w) \ \ (z\in \mathbb{C}).
		\end{equation}
		It is well known theorem that $ U_{\mu} $ is subharmonic on $ \mathbb{C} $, and harmonic on $ \mathbb{C}\setminus(supp \ \mu) $. Also \begin{equation}
			U_{\mu}(z)=\mu(\mathbb{C})\log|z|+O(|z|^{-1}) \ \ \  as \ \ \ z\rightarrow\infty.
		\end{equation}
		The logarithmic energy of $ \mu $ is defined by \begin{equation}
			I(\mu):=\int\int\log|z-w|d\mu(w)d\mu(z)=\int U_{\mu}(z)d\mu(z).
		\end{equation}
		It is possible that $ I(\mu)=-\infty $. A subset $ E $ of $ \mathbb{C} $ is called polar if $ I(\mu)=-\infty $ for every finite Borel measure $ \mu\neq 0 $ for which $ supp (\mu) $ is a compact subset of $ E $. Denote by $ \mathcal{P}(E) $ the collection of all Borel probability measures on $ E $. If there exists $ \omega_{E}\in \mathcal{P}(E) $ such that\begin{equation}
			I(\omega_E)=\sup_{\mu\in\mathcal{P}(E)} I(\mu),
		\end{equation}
		then $ \omega_E $ is called an equilibrium measure for $ E $. By \cite[Theorem 3.3.2]{ransfordbook}, every compact set $ E $ has an equilibrium measure. Moreover, if $ E $ is non-polar compact set, then this equilibrium measure is unique and $ supp(\omega_E) $ is subset of the exterior boundary of $ E $ (see eg. \cite[Theorem 3.7.6]{ransfordbook} ). 
		
		The Green's function for proper subdomain $ \Omega $ of $ \widehat{\mathbb{C}} $ is the non-negative subharmonic function which satisfy \begin{equation}
			g_{\Omega}(z)= U_{\omega_E}(z)-I(\omega_E)=U_{\omega_E}(z)-\log \mathrm{cap}(E)
		\end{equation}
		where $ \mathrm{cap}(E)=e^{I(\omega_E)} $ is the logarithmic capacity of $ E $.
		If $ \partial \Omega $ is non-polar, then there exists a unique Green's function $ g_{\Omega} $ for $ \Omega $ (\cite{ransfordbook}, Theorem 4.4.2).
		
		\subsubsection{Regular Measures}\label{extremalp} Given a Borel measure $\mu$ with compact support $$S_{\mu}:=supp(\mu)\subset \Bbb{C}$$ one can define the inner product 
$$\langle f,g \rangle:=\int_{\Bbb{C}}f(z)\overline{g(z)}d\mu $$ on the space of polynomials $\mathcal{P}_n.$ Then one can find uniquely defined orthonormal polynomials
 $$P_n^{\mu}(z)=\gamma_n(\mu)z^n+\cdots,\ \text{where}\ \gamma_n(\mu)>0\ \text{and}\ n\in\Bbb{N}.$$ Following \cite{general orth. polyn.}, we say that $\mu$ is \textit{regular}, denoted by $\mu\in \textbf{Reg}$, if
\begin{equation}
\lim_{n\to \infty}\gamma_n(\mu)^{1/n}=\frac{1}{\mathrm{cap}(S_{\mu})}.
\end{equation}
   
   The following result is well known (see \cite{duncan} and references therein). We provide a proof for the convenience of the reader.
\begin{prop}\label{ASMp}
Let $E\subset \Bbb{C}$ be a non-polar compact set and $\tau\in\textbf{Reg}$ be a finite measure supported on $E$. A sequence $p_{n}(z)=a_{n,n}z^n+\dots$ is asymptotically minimal on $E$ for $p=\infty$ if and only if it is asymptotically minimal with respect to $\tau$ and for some (equivalently for all) $p\in (0,\infty)$.	
\end{prop}
\begin{proof}
It is known that every measure $\tau\in \textbf{Reg}$ on $E$ satisfies Nikolski\u\i\ type inequality (cf. \cite[\S 3]{general orth. polyn.}) for some (equivalently for all) $p\in (0,\infty)$. Namely, there exists constants $M_n>0$ such that
$$\limsup_{n\rightarrow\infty}M_{n}^{1/n}=1$$ and for all polynomials $p_{n}$ whose degree is at most $n$ we have \begin{equation}\label{Neq}
		\|p_{n}\|_{E}\leq M_{n}\|p_{n}\|_{L^{p}(\tau)}. 
		\end{equation}
Recall that by minimality we always have 
$$\mathrm{cap}(E)^n\leq \|\frac{p_n}{a_{n,n}}\|_{E}.$$ 
Assume $\{p_{n}\}_{n\in\Bbb{N}}$ is asymptotically minimal on $E$ for some $p\in (0,\infty)$. Then, by (\ref{Neq}) we have \begin{equation}
		\mathrm{cap}(E)\leq\limsup_{n\rightarrow\infty}\dfrac{\|p_{n}\|_{E}^{1/n}}{|a_{n,n}|^{1/n}}\leq\lim_{n\rightarrow\infty}\dfrac{(M_{n}\|p_{n}\|_{L^{p}(\tau)})^{1/n}}{|a_{n,n}|^{1/n}}=\mathrm{cap}(E).
	\end{equation}
	This gives us \begin{equation}
		\lim_{n\rightarrow\infty}\dfrac{1}{n}\log\|p_n\|_{E}=0.	
		\end{equation}
		and hence $\{p_{n}\}_{n\in\Bbb{N}}$ be an asymptotically minimal on $E$ for $p=\infty$. 
	
	In order to prove reverse direction observe that
		$$\|p_n\|_{L^{p}(\tau)}\leq \|p_n\|_{E} \sqrt{\tau(E)}.$$
If  $\{p_{n}\}_{n\in\Bbb{N}}$ is asymptotically minimal with $p=\infty$ then 
		\begin{equation}
		\limsup_{n\rightarrow\infty}\dfrac{\|p_{n}\|_{L^{p}(\tau)}^{1/n}}{|a_{n,n}|^{1/n}}\leq \mathrm{cap}(E).
		\end{equation}
	Moreover, by (\ref{Neq}) we always have 	
	\begin{equation}
		\mathrm{cap}(E)\leq \liminf_{n\rightarrow\infty}\dfrac{\|p_{n}\|_{L^{p}(\tau)}^{1/n}}{|a_{n,n}|^{1/n}}.
		\end{equation}
Hence, $\{p_{n}\}_{n}$ is asymptotically minimal on $E$ with respect to $\tau$ and $p\in (0,\infty)$.  	
	\end{proof}

      Finally, we say that a sequence $\mu_n$ converges to a measure $\mu$ in the \textit{weak* topology} if for any continuous function $f$ with compact support in $\Bbb{C}$ we have $\int f d\mu_n\to \int f d\mu$. 

		\subsection{Polynomial Dynamics}
		
		Let $ p(z)=\sum_{j=0}^{n}a_{j}z^{j} $ be a polynomial of degree $ n\geq 2 $. The attracting basin of $ \infty $ for $ p$ is the set of points with unbounded forward orbit under $p$ that is\begin{equation}
			\Omega_p:=\{z\in\widehat{\mathbb{C}}: p^{k}(z)\rightarrow\infty \ \  as \ \  k\rightarrow\infty\} 	
		\end{equation}
		where $ p^{k}=\overbrace{p\circ\cdots \circ p}^{k \ many} $. Let us denote the escape radius for $p$ by $$R_p:=\frac{1+|a_n|+\dots+|a_0|}{|a_n|}.$$ It follows that 
		\begin{equation}
			\Omega_p=\bigcup_{k\geq0}p^{-k}(\mathbb{C}\setminus \overline{D(0,R_p)}).
		\end{equation}
		The set $ \Omega_{p} $ is a connected, open and completely invariant, i.e. $ p^{-1}(\Omega_p)=\Omega_{p} $. The complement $K_p:=\Bbb{C}\setminus \Omega_{p} $ is called the filled Julia set. The Julia set of $ p $ is defined by $ J_{p}:=\partial \Omega_{p}=\partial K_{p} $. It is well-known that $ K_{p} $ and $ J_{p} $ are compact and completely invariant. One can observe that \begin{equation}
			K_{p}=\bigcap_{k\geq0}p^{-k}(\overline{D(0,R_p)}).
		\end{equation}
		We denote the Dynamical Green's function of $p$ by
		$ g_{p}(z):\mathbb{C}\rightarrow[0,\infty) $  where
		\begin{equation}
			g_{p}=\lim_{k\rightarrow\infty}\dfrac{1}{n^{k}}\log^{+}|p^{k}(z)|
		\end{equation}
		where $\log^{+}=\max\{\log, 0\}$. By a theorem of Brolin \cite{brolin} the function $g_p$ coincides with the Green's function of $\Omega_p$ with the pole at infinity. Note that $ g_{p} $ vanishes precisely on $ K_{p} $ and has the invariant property
		\begin{equation}
			g_{p}(p(z))=n\cdot g_{p}(z). 
		\end{equation}	Moreover, the logarithmic capacity of the filled Julia set is given by
		\begin{equation}\label{Juliacap}
			\mathrm{cap}(K_{p})=\dfrac{1}{|a_{n}|^{\frac{1}{n-1}}}.
		\end{equation}
		It follows from \cite{brolin} that the measure $ \omega_p=\dfrac{1}{2\pi}\Delta g_{p} $ is $p$-invariant and coincides with the equilibrium measure of $ J_{p} $. The measure $\omega:= \omega_p$ is balanced i.e. for every set $ X\subset\widehat{\mathbb{C}} $ on which $ p $ is injective we have $ \omega(p(X))=n\cdot\omega(X) $. Moreover, Lyubich \cite{lyubich} proved that $ \omega$ is the unique measure of maximal entropy for $ p. $ 
		
		\section{Proof of Theorem \ref{mainteorem1}}\label{mainproof}
		Through this section, $ E, \Omega, \mathrm{Pc}(E), \omega_E$ and $p_n $ are as in the introduction. We also let $ g_{\Omega} $ be the Green's function for $ \Omega $ with pole at $ \infty $. We denote $K_{n}:=K_{p_n}$ the filled Julia set, $\Omega_{n}:=\Bbb{C}\setminus K_n$ and $J_{n}:=J_{p_n}$ denote the Julia set of $p_n$. We also denote the Brolin measures by $ \omega_{n}:=\omega_{J_{p_n}}$. Finally, we denote the counting measures of zeros of $p_n $ by 
		 \begin{equation}
			\mu_{n}:=\dfrac{1}{n}\sum_{p_{n}(z)=0}\delta_{z}
		\end{equation}
		where $ \delta_{z} $ denotes the unit mass at $z$.
		We need the following lemma on the behavior of $ |p_{n}|^{\frac{1}{n}} $. The proof is analogous to  that of \cite[Lemma 3.1 (i)]{LY}.
		\begin{lem}\label{Pn limiti ve potetial}
			Let $\{p_{n}\}_{n\in\Bbb{N}} $ be as in Theorem \ref{mainteorem1} with bounded zeros and $F\subset \Bbb{C}$ be a polynomially convex compact set such that  $$\displaystyle \cup_{n=1}^{\infty}\{z\in \Bbb{C}:p_n(z)=0\}\cup E\subset F$$ Assume that the sequence of counting measures of zeros $ \{\mu_{n}\}_{n\geq1} $ is weak*-convergent to a measure $ \nu $. Then $ |p_{n}|^{\frac{1}{n}}\rightarrow \exp(g_{\Omega})$ locally uniformly on $ \widehat{\mathbb{C}}\setminus F$.
		\end{lem}
		\begin{proof}
			Let $ \widetilde{p_{n}}:=\dfrac{p_{n}}{|a_{n,n}|} $. First, we will show that $ |\widetilde{p_{n}}|^{\frac{1}{n}}\rightarrow e^{U_{\nu}}$ 
			locally uniformly on $ \widehat{\mathbb{C}}\setminus F $.  Since the function $ w\rightarrow\log|z-w| $ is continuous on $ F $ for each $ z\in \mathbb{C}\setminus F $ we have 
			\begin{equation}
				\int_{F}\log|z-w|d\mu_{n}(w)\longrightarrow\int_{F}\log|z-w|d\nu(w) \ \ \ as \ \ n\rightarrow\infty.
			\end{equation} 
			By \cite[Lemma 3.3]{duncan} we have $supp(\nu)\subseteq \mathrm{Pc}(E)$. Thus, we obtain
			\begin{equation}
				\lim_{n\rightarrow\infty}\dfrac{1}{n}\log|\widetilde{p_{n}}(z)|=U_{\nu}(z),
			\end{equation}
			point-wise for all $ z\in \mathbb{C}\setminus F $. Now, let $ D $ be a disk whose closure is contained in $ \mathbb{C}\setminus F $. Since the sequence of analytic functions $ \{\widetilde{p_{n}}^{\frac{1}{n}}\}_{n\in\Bbb{N}} $ is uniformly bounded on $ D $, by Montel's Theorem every subsequence has a further subsequence converging uniformly to some analytic function $ f $  satisfying  $ |f|=e^{U_{\nu}} $ on $ D $. This shows that \begin{equation}
				\lim_{n\rightarrow\infty}|\widetilde{p_{n}}|^{\frac{1}{n}}= e^{U_{\nu}}
			\end{equation}
			locally uniformly on $ \mathbb{C}\setminus F $. Note that uniformly convergence near $ \infty $ follows from equation $ U_{\nu}(z)=\log|z|+o(|z|^{-1}) $ as $ z\rightarrow\infty $. Finally, using (\ref{definition asymp. min. eq. 1}) and by \cite[Theorem 4.7]{SaTo} we observe that 
			\begin{eqnarray}
				\lim_{n\rightarrow\infty}|p_{n}|^{\frac{1}{n}} &=& (1/\mathrm{cap}(E))e^{U_{\nu}}\\
				&=& \exp(g_{\Omega})
			\end{eqnarray} 
			locally uniformly on $ \widehat{\mathbb{C}}\setminus F $.
		\end{proof}
		It is known that the filled Julia sets of the sequence of orthonormal polynomials w.r.t a fixed regular measure on a compact set are uniformly bounded \cite{juliasetortpoly} (see also \cite{filled julia set of cheb. polyn} for Chebyshev polynomials). We prove the analogue result for asymptotically minimal polynomials:
		\begin{lem}\label{boundedness of Kn}
			Let $\{p_{n}\}_{n\in\Bbb{N}} $ be as in Theorem \ref{mainteorem1} with uniformly bounded zeros.  Assume that the counting measures of zeros $ \{\mu_{n}\}_{n\geq1} $ of $\{p_{n}\}_{n\in\Bbb{N}} $ is  weak*-convergent to a measure $ \nu $. Then there exists $ R>0 $ and $ N\in\Bbb{N} $ such that 
			\begin{equation}
				K_{n} \subset p_n^{-1}(\overline{\mathnormal{D}(0,R)})\subset \mathnormal{D}(0,R).
			\end{equation} for all $ n\geq N$.
		\end{lem}
		\begin{proof}
			For each $ \varepsilon> \log^+ \big(\mathrm{cap}(E)\big)$ there exists $ R>0 $ such that $$ \cup_{n=1}^{\infty}\{z\in \Bbb{C}:p_n(z)=0\}\cup E\subset D(0,\frac{R}{2})$$ and $ U_{\nu}(z)>2\varepsilon $ for all $ z\in\partial\mathnormal{D}(0,R) $. This implies that
			\begin{equation}
				U_{\nu}(z)-\log \mathrm{cap}(E)\geq \varepsilon
			\end{equation}
			for all $ z\in \partial \mathnormal{D}(0,R) $. Then by Lemma \ref{Pn limiti ve potetial} and compactness of $ \partial \mathnormal{D}(0,R) $, there exists $ N\in\Bbb{N} $ such that \begin{equation}
				|p_{n}(z)|^{\frac{1}{n}}\geq e^{\frac{\varepsilon}{2}}
			\end{equation}
			for all $ n\geq N $ and for all $ z\in \partial \mathnormal{D}(0,R) $. By increasing $ N $ if necessary, we can assume that $ \log(R)<N\dfrac{\varepsilon}{2} $. This gives $ |p_{n}(z)|>R $ for all $ z\in \partial \mathnormal{D}(0,R) $.\\
			Since the zeros of $p_n $ are contained in $  \mathnormal{D}(0,R) $, the minimum modules principle implies that \begin{equation}
				p_{n}(\mathbb{C}\setminus\mathnormal{D}(0,R))\subset \mathbb{C}\setminus\overline{\mathnormal{D}(0,R)}
			\end{equation} 
			for all $ n\geq N $. Thus, the equality $ K_{n}=\bigcap_{k\geq 0}p_{n}^{-k}(\overline{\mathnormal{D}(0,R)}) $ yields  \begin{equation}
				K_{n}\subset p_{n}^{-1}(\overline{\mathnormal{D}(0,R)})\subset \mathnormal{D}(0,R)
			\end{equation}
			for all $ n\geq N $.
		\end{proof}
		
		We will also need following result in the sequel:
		
		\begin{lem}\label{counting measure of pn-w}
			Let $\{p_{n}\}_{n\in\Bbb{N}} $ be as in Theorem \ref{mainteorem1}. For any compact set $ V\subseteq\mathbb{C}\setminus \mathrm{Pc}(E) $ we have 
			\begin{equation}\label{limit in the counting measure of pn-w}
				\lim_{n\rightarrow\infty}\big(\sup_{w\in D(0,R)}\big(\frac1n \sum_{p_n(z)=w}1_V(z)\big)\big)=0
			\end{equation}
			where $1_V$ denotes the characteristic function of the set $V$.
		\end{lem}
		In the proof of Lemma \ref{counting measure of pn-w}, we utilize the following result:
		\begin{lem}(\cite[Lemma 1.3.2]{general orth. polyn.})\label{rational funt bound by a}
			Let $ V, E\subseteq\mathbb{C} $ be two compact sets. If $V\subseteq\mathbb{C}\setminus \mathrm{Pc}(E) $ then there exists $ b<1 $ and $ N=N(E,V,b)\in\mathbb{N} $ such that for arbitrary $N$ points $ x_{1},\cdots,x_{N}\in V $ there exists $ N$ points $ y_{1},\cdots,y_{N}\in\mathbb{C}$ for which the rational function
			\begin{equation}
				r_{N}(z):=\prod_{j=1}^{N}\dfrac{(z-y_{j})}{(z-x_{j})}
			\end{equation}  
			has the sup norm $ \|r_{N}\|_{E}\leq b $. 
					\end{lem}
					
		\begin{proof}[\textbf{Proof of Lemma \ref{counting measure of pn-w}}] We prove the case $p\in (0,\infty)$ for a regular measure $\tau$. For $w\in D(0,R)$ we let $ x_{n,1},\cdots,x_{n,l(n)} $ be the roots of $ p_n-w $ in $V$. Then by Lemma \ref{rational funt bound by a} there exist $ b<1,\ N=N(E,V,b)\in\mathbb{N} $ and points  $ y_{n,1},\cdots, y_{n,l(n)}\in \mathbb{C}$ such that 
			\begin{equation}\label{normbound}
				\parallel (p_n(z)-w)\prod_{j=1}^{N{\lfloor\frac{l(n)}{N}\rfloor}}\dfrac{z-y_{n,j}}{z-x_{n,j}}\parallel_{L^{p}(\tau)}\leq \parallel p_n-w\parallel_{L^{p}(\tau)}b^{\lfloor\frac{l(n)}{N}\rfloor}.
				\end{equation}
				
				Now, let $ q_{n} $ be the monic polynomial of degree $n$ with minimal $ L^{p}(\tau)$ norm. Then by (\ref{normbound}) we obtain
				\begin{equation*}
					\|a_{n,n}q_{n}\|_{L^{p}(\tau)}\leq\|p_n-w\|_{L^{p}(\tau)}b^{\lfloor\frac{l(n)}{N}\rfloor}\leq (\|p_n\|_{L^{p}(\tau)}+R\|\tau\|^{\frac1p})b^{\lfloor\frac{l(n)}{N}\rfloor}
				\end{equation*}
				which in turn implies that
				\begin{equation}
					0\leq \frac1n\lfloor\frac{l(n)}{N}\rfloor\log\frac1b\leq \frac1n\log\bigg(\frac{\|p_n\|_{L^{p}(\tau)}+R\|\tau\|^{\frac1p}}{\|a_{n,n}q_n\|_{L^{p}(\tau)}}\bigg).
				\end{equation}
The right hand side is independent of $w$ and by asymptotic minimality of $p_n$ it tends to zero. Hence, the assertion follows.\\				
				The proof for the case $p=\infty$ is almost identical and omitted.
			\end{proof}
			
			The following lemma gives a relation between Green's functions for $ \Omega_{n} $ and $ \Omega $ with pole at infinity.
			\begin{lem}\label{greens functions}
				Let $\{p_{n}\}_{n\in\Bbb{N}} $ be as in Theorem \ref{mainteorem1}. Assume that the counting measures $ \{\mu_{n}\}_{n\geq1} $ of $\{p_{n}\}_{n\in\Bbb{N}} $ is  weak*-convergent to a measure $ \nu $. Then \begin{equation}\label{limit in green function}
					\limsup_{n\rightarrow\infty} g_{n}(z)\leq g_{\Omega}(z)
				\end{equation}
				locally uniformly on $ \mathbb{C} $ where $ g_{n} $ (resp. $ g_{\Omega} $) is the Green's function for $ \Omega_{n} $ (resp. $ \Omega $) with pole at infinity. 
			\end{lem}
			\begin{proof}
				By (\ref{definition asymp. min. eq. 1}) and (\ref{Juliacap}), we have $\displaystyle \lim_{n\rightarrow\infty}\mathrm{cap}(K_{n})=\mathrm{cap}(E)>0$. Thus, there exists $ C\in(0,1) $ such that $ \mathrm{cap}(K_{n})\geq C $ for $ n$ sufficiently large. Moreover, by Lemma \ref{boundedness of Kn} there exists $ R>0 $ and $ N\in\mathbb{N} $ such that $ K_{n}\subset p_n^{-1}(\overline{D(0,R)})\subset D(0,R) $ for all $ n\geq N $. Hence by \cite[Proposition 3.5]{filled julia set of cheb. polyn}, there exists $ N\in\mathbb{N} $ and $ M>0 $  such that \begin{equation}\label{greens functions, equation 1}
					\parallel g_{n}(z)-\dfrac{1}{n}\log^{+}|p_{n}(z)|\parallel_{\infty}\leq\dfrac{M}{n}
				\end{equation}
				for all $ n\geq N $.
				
				Now, let $ p\in(0,\infty) $. Then by \cite[Theorem 3.2.1]{general orth. polyn.} and \cite[Remark 3.2.2]{general orth. polyn.} 
				\begin{equation}
					\limsup_{n\rightarrow\infty}\bigg(\dfrac{|p_{n}(z)|}{\parallel p_n\parallel_{L^{p}(\tau)}}\bigg)^{\frac{1}{n}}\leq e^{g_{\Omega}(z)}
				\end{equation}
				locally uniformly for $ z\in\mathbb{C} $. This implies that 
				\begin{equation}\label{greens functions, equation 3}
					\limsup_{n\rightarrow\infty} (|p_{n}(z)|)^{\frac{1}{n}}\leq e^{g_{\Omega}(z)}
				\end{equation}
				locally uniformly on $\mathbb{C} $. 
				
				For $ p=\infty$, using (\ref{definition asymp. min. eq. 2}) and Bernstein's Lemma (see \cite[Theorem 5.5.7]{ransfordbook}), we obtain  \begin{equation}\label{greens function, equation 4}
					\limsup_{n\rightarrow\infty} (|p_{n}(z)|)^{\frac{1}{n}}\leq e^{g_{\Omega}(z)}
				\end{equation}
				uniformly on $ \mathbb{C} $. Hence, the assertion follows by combining (\ref{greens functions, equation 1}), (\ref{greens functions, equation 3}) and (\ref{greens function, equation 4}). 
			\end{proof}
The next example illustrates that we need both the convergence of $n^{th}$ roots of the leading coefficients and the corresponding norms in order to have convergence of the corresponding Brolin measures or filled Julia sets in the suitable sense: 			
	\begin{example}\label{Ex}
Let $p_n(z)=2^{n^2}z^n$. Then $p_n(z)/a_{n,n}=z^n$ is minimal on the unit circle but the filled Julia sets $K_n$ of $p_n$ decreases to the point at the origin.	
	\end{example}		
			\begin{proof}[\textbf{Proof of the Theorem \ref{mainteorem1}}]
				First, we will show that $\{\omega_{n}\}_{n\in\Bbb{N}} $ is sequentially pre-compact with respect to the weak$^{*} $-topology. Indeed, for each subsequence $ \{\omega_{n_{k}}\}_{k\in\Bbb{N}}$, since support of $ \mu_{n_{k}} $'s are uniformly bounded, by Helly's Theorem 
				there is a further subsequence 	such that the empirical measure of zeros $ \mu_{n_{k_l}} $ weak* converges to a probability measure $\nu$. Then by Lemma \ref{boundedness of Kn}, $ K_{n_{k_{l}}} $'s are uniformly bounded and hence the Brolin measures $\{ \omega_{n_{k_{l}}}\}_{l\geq1} $ has a  weak* convergent subsequence. 
				
				Next, by  Brolin's Theorem \cite{brolin}, for any measurable function $ f:\mathbb{C}\rightarrow\mathbb{C} $ and for all $ n\in\mathbb{N} $, \begin{equation}
					\int_{\mathbb{C}}f(z)d\omega_{n}(z)=\dfrac{1}{n}\int_{\mathbb{C}}\bigg(\sum_{p_{n}(w)=z}f(w)\bigg)d\omega_{n}.
				\end{equation}
				Now, let $ \sigma $ be a weak* limit of a subsequence $ \{\omega_{n_{k}}\}_{k\in\Bbb{N}}$. We will show that $ supp(\sigma)\subset J=\partial \mathrm{Pc}(E) $. Passing to a further subsequence if necessary, we may and we do assume that the normalized measure of zeros $\mu_{n_{k}}$ are weak* convergent to a measure $\nu$.  Let $ V\subset\mathbb{C} $ be a compact set with $ V\cap \mathrm{Pc}(E)=\emptyset $. Then, by Lemma \ref{boundedness of Kn} the filled Julia sets $K_n$ are uniformly bounded and by Lemma \ref{counting measure of pn-w} we have
				\begin{equation}
					\omega_{n_{k}}(V)=\int_{\mathbb{C}}1_{V}(z)d\omega_{n_{k}}(z)=\int_{\mathbb{C}}\bigg(\dfrac{1}{n_{k}}\sum_{p_{n_{k}}(w)=z}1_{V}(w)\bigg)d\omega_{n_{k}}(z)\xrightarrow{k\rightarrow\infty}0.
				\end{equation}
				Hence, $ \sigma(V)=0$. Since this is true for all compact set disjoint from $ \mathrm{Pc}(E) $, we deduce that $ supp(\sigma)\subset \mathrm{Pc}(E) $.
				
				On the other hand, by Lemma \ref{greens functions} we have that $\displaystyle  \lim_{n\rightarrow\infty}g_{n}=0 $ uniformly on any compact subset of the interior of $ \mathrm{Pc}(E) $ since $ g_{\Omega}\equiv 0 $ in the interior of $ \mathrm{Pc}(E) $. 
				
						Let $ L $ be a compact subset of the interior $Int(\mathrm{Pc}(E)) $ and $ U$ be a compact neighborhood of $ L $ contained in $Int(\mathrm{Pc}(E))$ and $ \varphi $ be a $ C^{2} $ function with support in $ supp(\varphi)\subset U$ satisfying $ 0\leq \varphi \leq 1$ and $ \varphi\equiv1 $ on $ L $. Then we have
						\begin{eqnarray*}
							\sigma(L) &\leq & \int\varphi(z)d\sigma(z)\\ 
							&=& \lim_{k\rightarrow\infty}\int\varphi(z)d\omega_{n_{k}}(z)\\
							&=& \lim_{k\rightarrow\infty}\dfrac{1}{2\pi}\int \Delta\varphi(z)g_{n_{k}}(z)dA(z)=0
						\end{eqnarray*}
						where $ dA(z) $ is the standard Euclidean area element.
						
						On the other hand, by \cite[Lemma A.3.3]{ransfordbook} 
						$$\sigma(Int(\mathrm{Pc}(E)))=\sup\{\sigma(L): L\subset Int(\mathrm{Pc}(E))\ \text{and}\ L\ \text{is compact}\}$$ 
						Hence, we conclude that $ \sigma(Int(\mathrm{Pc}(E))=0 $ and $ supp(\sigma)\subset \partial \mathrm{Pc}(E)$.
						
						Finally, we will show that $  \omega_{n} \ \xrightarrow{weak^{*}} \ \omega_E  $ as $ n\rightarrow\infty $. By \cite[Lemma 3.3.3]{ransfordbook}, we have 
						\begin{equation}
							\limsup_{k\rightarrow\infty}\mathnormal{I}(\omega_{n_{k}})\leq \mathnormal{I}(\sigma).
						\end{equation}
						On the other hand, \begin{equation}
							\limsup_{n\rightarrow\infty}\mathnormal{I}(\omega_{n_{k}})=\mathnormal{I}(\omega_E)
						\end{equation}
						by asymptotic minimality \ref{definition asyp. min.}  and (\ref{Juliacap}) this in turn implies that $ \mathnormal{I}(\omega_E)\leq\mathnormal{I}(\sigma) $. Since $ \omega_E $ is the unique measure of maximal energy, we deduce that $ \sigma=\omega_E $. Since this is true for all weak*-convergent subsequences of $ \{\omega_{n}\}_{\geq1} $ we conclude that $\omega_{n} \to \ \omega_E$ in the weak* topology.
					\end{proof}
					
					\subsection{Examples}\label{examples}
		First, we observe that there are sequences of asymptotically minimal polynomials such that the zeros are not uniformly bounded. In particular, the assumption in Theorem \ref{mainteorem1} on the zeros can not be removed. 
\begin{example}\label{cexample}
Let $c_n\in \Bbb{C}$ be a sequence such that $|c_n|\to \infty$ and $\frac1n\log |c_n|\to 0$ as $n\to \infty$. Let also $$p_n(z)=z^{n-1}(z-c_n)\ \ \text{for}\ n\geq 1.$$ Then it is easy to see that $\{p_n\}_{n\geq 1}$ is asymptotically minimal on the unit circle $S^1$ with respect to the equilibrium measure $\frac{d\theta}{2\pi}$. On the other hand, $c_n\in K_n$ and hence $\{K_n\}_{n\geq 1}$ is not uniformly bounded in $\Bbb{C}$.
\end{example}					
					Next, we review natural classes of polynomials that fit into the framework of Theorem \ref{mainteorem1}. Note that orthonormal polynomials associated with a regular measure and normalized Chebyshev polynomials on a non-polar compact set are standard examples of extremal polynomials.
				\begin{example}[$L^p$-minimal polynomials]
			Let $E\subset \Bbb{C}$ be a non-polar compact set and the measure $\tau\in\textbf{Reg}$. For $p\in [1,\infty]$ there exists unique polynomials $p_n$ with minimal $L^p(\tau)$-norm satisfying 
			$$\lim_{n\to\infty}\frac1n\log\|p_n\|_{L^p(\tau)}=\log \mathrm{cap}(E).$$ Then the sequence $\{\frac{p_n}{\|p_n\|_{L^p(\tau)}}\}_{n\geq 1}$ is asymptotically minimal on $E$ (\cite[\S 3]{general orth. polyn.}). We remark that the case $p=\infty$ corresponds to normalized Chebyshev polynomials.
				\end{example}		
						\begin{example}[Fekete Polynomials]
							Let $ E $ be a non-polar compact set. For $ n\geq2 $, we denote $n$-tuple of points $w_{1},\cdots,w_{n}\in E $  such that the supremum $\sup\{\prod_{j<k}|z_{j}-z_{k}|^{\frac{2}{n(n-1)}}\} $ is attained at these points. The points $F_n:=(w_{1},\cdots,w_{n})$  are called Fekete points of order $n$. We remark that Fekete points are not unique in general but do exists by compactness of $E$. The Fekete polynomial associated with $F_n$ is given by
							\begin{equation}
								q_{n}(z)=\prod_{j=1}^{n}(z-w_{j})
							\end{equation}
							By \cite[\S5]{ransfordbook} 
							\begin{equation}\label{example 1, equation}
								\lim_{n\rightarrow\infty}\| q_{n}\|_E^{\frac{1}{n}}=\mathrm{cap}(E).
							\end{equation} Thus, the normalized sequence $p_n=\dfrac{q_{n}}{\|q_{n}\|_E}$ is asymptotically minimal on $ E $. Note that by construction all the zeros of $p_n$ lie in $E$.
						\end{example}

						\begin{example}[Faber Polynomials]
							Let $ E $ be a non-polar compact set such that the unbounded component $ \Omega $ of $ \widehat{\mathbb{C}}\setminus E $ is simply connected. Let $ \phi:\Omega\rightarrow\widehat{\mathbb{C}}\setminus\{z: \parallel z\parallel\leq1\} $ be the (unique) conformal map such that $ \phi(\infty)=\infty $ and $ \phi'(\infty)>0 $. It is well known that \begin{equation}
								\phi(z)=\dfrac{z}{cap (E)}+a_{0}+\dfrac{a_{1}}{z}+\cdots .
							\end{equation}
							The Faber polynomial $F_{n}$ of degree $n$ is defined by the equation
							\begin{equation}
								F_{n}(z)=\phi(z)^{n}+\mathcal{O}(\frac{1}{z})  \ \ \  z\rightarrow\infty.
							\end{equation}
							These polynomials satisfy (\ref{example 1, equation}) (see \cite{defin faber poly}). Thus, the normalized sequence\\ $p_{n}=\dfrac{F_{n}}{\| F_{n}\|_E}$ is asymptotically minimal on $ E $. Furthermore, if $ E $ is convex then all zeros of $p_n $ lies in the interior of $ E $ (see \cite[Theorem 2]{zeros of faber poly}). Hence, Theorem \ref{mainteorem1} applies.
						\end{example}
						
						\begin{example}[Polynomials with bounded coefficients] 	
						Let				
				$$p_n(z):=z^n+a^n_{n-1}z^{n-1}+\dots+ a^n_0$$ where $|a^n_j|< M$ for some fixed $M>0$. Then it is easy to see that $\{p_n\}_{n\in\Bbb{N}}$ is asymptotically minimal on the unit circle $S^1:=\{z\in\Bbb{C}: |z|=1\}$ with respect to the equilibrium measure $\frac{1}{2\pi}d\theta$. Moreover, Cauchy bounds imply that if $p_n(z)=0$ then $|z|\leq 1+\max_{0 \leq j\leq n-1}|a_j^n|$. In particular, zeros of $p_n$ are uniformly bounded and contained in the disc $D(0,M+1)$. Then it follows from Theorem \ref{mainteorem1} that their Brolin measures $\omega_n\to \frac{1}{2\pi}d\theta$ in the weak* topology.
						
						\end{example}
						
						\section{Geometric Limit of Filled Julia Sets}
						
						In this section, we focus on approximating planar regular compact subsets by filled Julia sets of asymptotically minimal polynomials. To this end, we consider Klimek and Hausdorff distances on such sets. 						
						\subsection{Topology of Compact Sets}
						
						We denote the collection of all non-empty compact subsets of $\Bbb{C}$ by $\mathscr{K}$. The classical Hausdorff metric $\chi$ on $\mathscr{K}$ is defined by  						
						\begin{equation*}
							\chi(A,B)=max(h(A,B),h(B,A))
						\end{equation*} where 
						\begin{equation*}
							h(A,B)=\sup_{a\in A}\inf_{b\in B} \parallel a-b\parallel.
						\end{equation*}
						The pair $(\mathscr{K},\chi)$ forms a complete metric space. If $ \{A_{n}\}_{n\in\Bbb{N}}\subset\mathscr{K} $ is a uniformly bounded sequence of compact sets, we define the sets 
						\begin{equation}
							\liminf_{n\rightarrow\infty} A_{n}:=\{z\in\mathbb{C}:\exists\{z_{n}\}, A_{n}\ni z_{n}\xrightarrow{n\rightarrow\infty}z\},
						\end{equation}
						and
						\begin{equation}
							\limsup_{n\rightarrow\infty} A_{n}:=\{z\in\mathbb{C}:\exists\{n_{k}\}, n_{k}\nearrow \infty \ and \ \exists\{z_{n_{k}}\}, A_{n_{k}}\ni z_{n_{k}}\xrightarrow{k\rightarrow\infty}z\}.
						\end{equation}
						It is easy to see that a uniformly bounded sequence $\{A_n\}_{n\in\Bbb{N}}\in \mathscr{K}$ is pre-compact in $(\mathscr{K},\chi)$ and the sets $\displaystyle \liminf_{n\rightarrow\infty}A_{n}$ and $\limsup_{n\rightarrow\infty}A_{n} \in \mathscr{K}$ (cf. \cite{juliasetortpoly}, Lemma 3.1). Moreover, $ \{A_{n}\}_{n\in\Bbb{N}} $ converges to $ A $ with respect to Hausdorff metric, denoted by $\displaystyle \lim_{k\rightarrow\infty} A_{n}= A $, if and only if \linebreak $\displaystyle\liminf_{n\rightarrow\infty} A_{n}= \limsup_{n\rightarrow\infty} A_{n}=A $. 						
						
						Next, we denote the collection of all polynomially convex compact regular subsets of $\Bbb{C}$ by $\mathscr{R}$. M. Klimek \cite{Klimek} defined a natural metric by using Green's functions: for regular compact subsets $ E, F $ of $ \mathbb{C} $, we let $g_{\Omega_{E}}, g_{\Omega_{F}} $ be Green's functions with the pole at infinity for $ \mathbb{C}\setminus \mathrm{Pc}(E)$ and $ \mathbb{C}\setminus \mathrm{Pc}(F) $ respectively. The Klimek distance between $ E $ and $ F $ is defined by 
						\begin{equation}\label{Gamma}
							\Gamma(E,F):=max(\parallel g_{\Omega_{E}}\parallel_{F},\parallel g_{\Omega_{F}}\parallel_{E})=\parallel g_{\Omega_{E}}-g_{\Omega_{F}}\parallel_{\mathbb{C}}. 
						\end{equation}
						The Klimek distance $\Gamma$ induces a psuedo-metric on regular compact subsets of $ \mathbb{C} $ and a metric on $\mathscr{R}$. In fact, the pair $(\mathscr{R},\Gamma)$ forms a complete metric space \cite{Klimek}.					
						
						We remark that the topologies induced by Hausdorff and Klimek metrics on $\mathscr{R}$ are different. In particular, convergence in Klimek distance does not imply convergence in Hausdorff distance (see Example \ref{dynamicalexample} below).
						On the other hand, there is a relation between these two metrics for some special cases. Klimek showed that if $ E_{n},E $ are regular, polynomially convex, connected subsets of the complex plane containing the origin, and if $ \{ E_{n}\}_{n\in\Bbb{N}} $ converges to $ E $ in Hausdorff topology then $ \{E_{n}\}_{n\in\Bbb{N}} $ converges to $ E $ in Klimek distance (\cite[Proposition 1]{Klimek}). 
						
						Recall that for $E\in \mathscr{K}$ and $\delta>0$ the \textit{modulus of continuity} is defined by 
						$$\omega_E(\delta)=\sup\{g_{\Omega_{E}}(z): dist(z,E)\leq \delta\}.$$ In particular, $E\in\mathscr{K}$ is regular if and only if $\lim_{\delta\to 0^+}\omega_E(\delta)=0$. In general convergence in Hausdorff metric does not imply convergence in Klimek metric but the following is known:			
						
						\begin{prop}\cite{Siciak2}\label{convergence}
							A subfamily $\mathcal{E}\subset \mathscr{R}$ is pre-compact with respect to Klimek topology if and only if the following hold
							\begin{enumerate}
								\item The family $\mathcal{E}$ is uniformly bounded that is there exists $R>0$ such that $E\subset B(0,R)$ for all $E\in \mathcal{E}$ 
								\item The family $\mathcal{E}$ has equicontinuity property that is $\displaystyle\lim_{\delta\to 0^+}[\sup_{E\in\mathcal{E}}\omega_E(\delta)] =0$ 
							\end{enumerate}
						\end{prop}						
						In particular, if $E_n\in \mathscr{R}$ and $\chi(E_n,E)\to0$ for some $E\in \mathscr{R}$ then $\Gamma(E_n,E)\to0$ if and only if $\displaystyle\lim_{\delta\to 0^+}[\sup_{n}\omega_{E_n}(\delta)] =0$. 						
						
						\subsection{Convergence of Filled Julia Sets}
						Through this section, let $ E $ be a compact non-polar subset of $ \mathbb{C} $, $ \Omega $ be the unbounded  component of $ \mathbb{C}\setminus E $ and $ g_{\Omega} $ be the Green's function for $ \Omega $ with pole at $ \infty $. Moreover,  let $ \{p_{n}\}_{n\geq 1} $ be an asymptotically minimal on $ E $, we let $ \Omega_{n},J_{n},K_{n} $ and $ g_{n} $ denote the basin of attraction for $ \infty $, the Julia set, the filled Julia set of $p_n $ and the dynamical Green's function for $p_n $ respectively. 
						
		For $ s>0 $ and a compact set $ E\subset\mathbb{C}$ we denote by $ E^{s}:=\{z\in\mathbb{C}: dist(z,E)<s\} $. Clearly, $ E^{s}$ forms a neighborhood base of the set $E$ in $\Bbb{C}$. Furthermore, if the compact set $ E $ is regular, we define $E_{s}:=\{z\in\mathbb{C}:g_{\Omega}(z)\leq s\}  $ and $ \Omega_{s}:=\{z\in\mathbb{C}: g_{\Omega}(z)>s\} $. It follows that $ (E_{s})_{s>0} $ also form a neighborhood base of $ \mathrm{Pc}(E) $ (see \cite{Klimek}). 
		
		Now, we prove the main result of this section:						

\begin{proof}[Proof of Theorem \ref{approximation by filled julia sets}]
First, we will show that $\{ K_{n}\}_n$ is sequentially pre-compact in $(\mathscr{R},\Gamma)$. Indeed, passing to a subsequence $\{K_{n_j}\}_{j}$ we may assume that the counting measures of zeros of $ p_{n_j} $'s are weak*-convergent. Then by Lemma \ref{boundedness of Kn} the collection $\{K_{n_j}\}_j$ is uniformly bounded. Next, we show that $\{K_{n_j}\}_j$ has the equicontinuity property. Indeed, for $s>0$ by assumption  all zeros of $p_n$ are contained in $E_{\frac{s}{2}}$ for sufficiently large $n$. Then it follows from Lemma \ref{Pn limiti ve potetial}  							
							\begin{equation}
								|\frac{1}{n_j}\log |p_{n_j}(z)|-g_{\Omega}(z)|<\dfrac{s}{2}  \ \text{for} \  z \in\partial\Omega_{s}.	
							\end{equation}							
							for sufficiently large $n_j$. This in turn implies that
							\begin{equation}
								|p_{n_j}(z)|>e^{\frac{sn_j}{2}} \ \text{for} \  z \in\partial\Omega_{s}.	
							\end{equation}
Since all zeros of $p_{n_j}$ are contained in $\Bbb{C}\setminus\Omega_s$ by applying the minimum modulus principle on the domain $\Omega_s$ we deduce that $|p_{n_j}(z)|>e^{\frac{sn_j}{2}}$ for all
							$z\in \Omega_s$ and sufficiently large $n_j$. Next, by (\ref{greens functions, equation 1}) there exists $M>0$ such that
							\begin{equation}
							|g_{n_j}(z)-\frac1n\log|p_{n_j}(z)||\leq \frac{M}{n_j}\ \text{for}\ z\in \Omega_s
							\end{equation}
which implies that $|g_{n_j}(z)|>\frac{s}{4}$ for $z\in \Omega_s$ and sufficiently large $n_j$. This in turn yields

							\begin{equation}\label{include} K_{n_j}\subset \Bbb{C}\setminus\Omega_s=E_{s} \end{equation} for sufficiently large $n_j$. Since $(E_s)_{s>0}$ form a neighborhood bases for $\mathrm{Pc}(E)$ in $\Bbb{C}$ for each $\epsilon>0$ we can find $0<s<\epsilon$ such that 
							$E_{s}\subset \mathrm{Pc}(E)^{\epsilon}$. 
							On the other hand, by Lemma \ref{greens functions}, we have 
							\begin{equation}\label{ineq2}
								g_{n_j}(z)\leq g_{\Omega}(z)+s
							\end{equation}
							for all $ z\in \mathrm{Pc}(E)^{\epsilon}$ and sufficiently large $n_j$. 
							Hence from (\ref{include}) and (\ref{ineq2}) we deduce that 
							$$\lim_{\epsilon\to 0^+}[\sup_{n_{j}}\omega_{K_{n_j}}(\epsilon)] =0. $$
							Thus, by Proposition \ref{convergence} the family $\{ K_{n}\}_{n\in\Bbb{N}}$ is sequentially pre-compact in $(\mathscr{R},\Gamma)$. Moreover, $\Gamma(K_{n_j},\mathrm{Pc}(E))\to 0$. Indeed, for each $s>0$ by (\ref{include})
							\begin{equation}\label{ineq1}
								\|g_{\Omega}\|_{K_{n_j}}\leq s
							\end{equation} for sufficiently large $n_j$. Moreover, since $E$ is regular $ g_{\Omega}(z)=0 $ on $ \mathrm{Pc}(E)$ and by (\ref{ineq2}) we conclude that $ \parallel g_{n_j}\parallel_{\mathrm{Pc}(E)}\leq s $. Hence, combining (\ref{ineq1}) and (\ref{ineq2}) we deduce that
							 $$ \Gamma(K_{n_j},\mathrm{Pc}(E))=\max(\parallel g_{n_j}\parallel_{\mathrm{Pc}(E)},\parallel g_{\Omega}\parallel_{K_{n_j}})\leq s.$$ 
							 Since $s>0$ arbitrary we conclude that $\Gamma(K_{n_j},\mathrm{Pc}(E))\to 0$ as $j\to\infty$. 
							
							For the general case, since the zeros of $ \{p_{n}\}_{n} $ are bounded for each Klimek convergent subsequence of $ \{K_{n}\}_{n} $ it has a further subsequence $\{K_{n_j}\}$ such that the counting measures of zeros of corresponding $p_n $'s are weak*-convergent. Then by above argument $\Gamma(K_{n_j},\mathrm{Pc}(E))\to0$. Since this holds for all convergent subsequences we conclude that $ \{K_{n}\}_{n\in\Bbb{N}} $ converges to $ \mathrm{Pc}(E) $ in the Klimek distance.	
						\end{proof}
						As a corollary we obtain the following:
			\begin{cor}
			The collection of all filled Julia sets of asymptotically minimal polynomials associated with regular planar compact sets is a proper dense subset of $(\mathscr{R},\Gamma)$.
			\end{cor}		
			\begin{proof}
			Note that the filled Julia set of a polynomial of degree $d\geq 2$ has H\"{o}lder property \cite{CG}. Hence, the collection of filled Julia sets is a proper subset of $\mathscr{R}$. The density follows from Theorem \ref{approximation by filled julia sets}.
			\end{proof}
						The next example illustrates that for a sequence of asymptotically minimal polynomials $ \{p_{n}\}_{n} $ associated with a regular compact set $E\subset \Bbb{C}$ as in Theorem \ref{approximation by filled julia sets}, their filled Julia sets need not to converge in the Hausdorff topology.  It was stated as an open problem in \cite{filled julia set of cheb. polyn} that for the sequence of dual Chebyshev polynomials and a limit $K_{\infty}$ set of $\{K_n\}_n$ in $(\mathscr{K}, \chi)$ whether the difference $E\setminus K_{\infty}$ is a polar set. For asymptotically minimal polynomials this difference could be quite large:	
						
						\begin{example}\label{dynamicalexample}
							For fixed $c\in \Bbb{C}$, we let $p_n(z)=z^{n}+c $ for $ n\in\Bbb{N} $. Then it is easy to see that $ \{p_{n}\}_{n\in\Bbb{N}} $ satisfy hypotheses of Theorem \ref{approximation by filled julia sets} on the unit circle $ E=S^1$. Then, the filled Julia sets of $p_n $'s converge to $ S^1 $ in the Klimek distance. On the other hand, by \cite[Theorem 1.2]{pn+c example} if $ |c|<1 $ then the filled Julia sets $K_{p_n}$ converges to the closed unit disc $\overline{\mathbb{D}}$; however if $ |c|>1 $ then the filled Julia sets converges to $S^1$ with respect to Hausdorff topology. Finally, for almost every $c\in S^1$ the filled Julia sets do not converge to any compact set \cite{KRS}. Moreover, again by \cite[Theorem 1.2]{pn+c example} for any limit set $K_{\infty}$ of the filled Julia sets $K_{p_n}$ in the Hausdorff topology we have $\mathrm{Pc}(K_{\infty})=\overline{\mathbb{D}}$
						\end{example}
						
Recall that in section \S \ref{mainproof}, we showed that if the counting measures of zeros of $ \{p_{n}\}_{n} $ are weak*-convergent, then the filled Julia sets $ K_{n} $'s are uniformly bounded. This yields us $ \{K_{n}\}_{n} $ is sequentially pre-compact in  $ \mathscr{K} $ with respect to Hausdorff topology. Motivated by Example \ref{dynamicalexample} we prove following result:							
						\begin{prop}\label{limitsets}
							Let  $ E $ be a regular compact set in $ \mathbb{C} $ and $ \{p_{n}\}_{n\in\Bbb{N}} $ be an asymptotically minimal sequence. Assume that for all $ \varepsilon>0 $ there exists $ N\in\mathbb{N} $ such that the zeros of $p_n $'s are contained in $ \mathrm{Pc}(E)^{\varepsilon} $ for all $ n\geq N $.  Let $ K_{\infty}\subset \Bbb{C}$ be a compact set that is a limit  point of $\{K_n\}_{n\in\Bbb{N}}$ with respect to Hausdorff metric then $ \mathrm{Pc}(K_{\infty})=\mathrm{Pc}(E)$.
						\end{prop}						
				First, adapting the argument in \cite[Proposition 4.3]{filled julia set of cheb. polyn} we can prove the following:
						
						\begin{lem}\label{E subset of liminf limsup }
							For any limit point $ K_{\infty} $ of a convergent subsequence $ \{K_{n_{k}}\}_{k} $ with respect to Hausdorff topology we have that \begin{equation}
								E\subset \mathrm{Pc}(K_{\infty})\subset \mathrm{Pc}(\limsup_{n\rightarrow\infty}K_{n}).
							\end{equation}
							
						\end{lem}	
						
						\begin{proof}
							Passing to a subsequence if necessary we may assume that the counting measures of zeros of $ \{p_{n}\}_{n} $ are weak*-convergent. The rest of the proof follows from \cite[Proposition 4.3 ]{filled julia set of cheb. polyn} and Lemma \ref{greens functions}. 
						\end{proof}	
				Now, we prove Proposition \ref{limitsets}:			
						\begin{proof}[Proof of Proposition \ref{limitsets}]
							Note that $ E_{s} $ is polynomially convex for all $ s>0 $ \cite{extremal plrsubhrmnic fnction}. Passing to a subsequence if necessary we may assume that the counting measures of zeros of $ \{p_{n}\}_{n} $ is weak*-convergent. Then, by Lemma \ref{E subset of liminf limsup } and the proof of Theorem \ref{approximation by filled julia sets}, for $ s>0 $ small
							we have 
							\begin{equation}
								E\subset \mathrm{Pc}(K_{\infty})\subset \mathrm{Pc}(\limsup_{n\rightarrow\infty}K_{n})\subset E_{s}.
							\end{equation}
							Letting $s\to0$ we deduce that
							$$\mathrm{Pc}(K_{\infty})= \mathrm{Pc}(E).$$
						\end{proof}						
						Finally, we focus on the Hausdorff limit of the Julia sets of asymptotically minimal polynomials.  Let $ E $ be a compact non-polar subset of $ \mathbb{C}$ and $\Omega$ be the unbounded component of $\Bbb{C}\setminus E$. We denote the outer boundary of $E$ by $ J_E:=\partial \Omega$. We also denote the exceptional set (see \cite{ransfordbook} for definition) for the Green's function $ g_{\Omega} $ by $F_E$. This means that $ F_E=\{z\in E: g_{\Omega}(z)>0\} $. We adapt the argument in \cite[Theorem 1.3(ii)]{juliasetortpoly} to our setting to prove that the limit of Julia sets of asymptotically minimal polynomials contain the regular points of the outer boundary:

\begin{thm}
Let $ E $ be a compact non-polar subset of $ \mathbb{C} $ and $ \{p_{n}\}_{n\in\Bbb{N}} $ be a sequence of asymptotically minimal polynomials whose zeros are contained in $ \mathrm{Pc}(E) $. Then,
							\begin{equation}
								\overline{J\setminus F}\subseteq\liminf_{n\rightarrow\infty}J_{n}.
							\end{equation}
							
In particular, if $ J $ is regular, then \begin{equation}
								J\subseteq\liminf_{n\rightarrow\infty} J_{n}.
							\end{equation} 
\end{thm}
\begin{proof}[Sketch of Proof] Since we mainly follow the argument in the proof of \cite[Theorem 1.3(ii)]{juliasetortpoly} we only give the main differences that require clarification.
Assume that there exists $ z_{0}\in J\setminus F $ such that $ z_0\notin \displaystyle\liminf_{n\rightarrow\infty}J_{n} $. Then $ g_{\Omega}(z_{0})=0 $ and there exists $ \delta>0 $ and $ (n_{k})_{k} $ with $ n_{k}\nearrow \infty $ such that for all $ k $, $ D(z_{0},\delta)\cap J_{n_{k}}=\emptyset$. By passing to a further subsequence if necessary we may assume that the counting measures of zeros of $ p_{n_{k}} $'s are weak*-convergent. Then using Lemma \ref{Pn limiti ve potetial} and inequality (\ref{greens functions, equation 1}), we can observe that for every compact set $ V\subseteq\Omega $ and every $ \varepsilon>0 $, we have \begin{equation}\label{last theorem,limif Jn,equation 1}
								\lim_{n\rightarrow\infty}\mathrm{cap}(\{z\in V:g_{\Omega_{n_{k}}}(z)<g_{\Omega}(z)-\varepsilon\})=0.
							\end{equation}
							Then following the argument in the proof of \cite[Theorem 1.3(ii)]{juliasetortpoly} one can show that there exist $ \varepsilon>0 $ and $ k\in\mathbb{N} $ such that for all $ k\geq N $, there exists $ z_{k}\in  D(z_{0},\delta) $ with $ g_{n_{k}}(z_{k})\geq\epsilon $ and $ D(z_{0},\delta)\subset\Omega_{n_{k}} $.  By Harnack's inequality, we obtain \begin{equation}
								g_{n_{k}}(z_{0})\geq \frac13g_{n_{k}}(z_{k})\geq\dfrac{\epsilon}{3}>0.
							\end{equation}
							On the other hand, by Lemma \ref{greens functions}, we have that \begin{equation}
								\limsup_{k\rightarrow\infty}g_{n_{k}}(z_{0})\leq g_{\Omega}(z_{0})=0
							\end{equation}
							which is a contradiction. Hence, we deduce that \begin{equation}
								\overline{J\setminus F}\subseteq\liminf_{n\rightarrow\infty}J_{n}.
							\end{equation}
							
\end{proof}
						
		\section*{Acknowledgement}
We are grateful to the anonymous referees for their comments which improve the presentation of this paper.


\begin{thebibliography}{XX}
							
							\bibitem{BKS}Bialas-Ciez, L., Kosek, M.,  Stawiska, M. On Lagrange polynomials and the rate of approximation of planar sets by polynomial Julia sets. J. Math. Anal. Appl., 464(1), 507-530 (2018).
										
							
							\bibitem {BP} Bishop, C. J., Pilgrim, K. Dynamical dessins are dense. Rev. Mat. Iberoam., 31(3), 1033-1040 (2015).
							
							\bibitem{pn+c example}Boyd, S. H.,  Schulz, M. J. Geometric limits of Mandelbrot and Julia sets under degree growth. International Journal of Bifurcation and Chaos, 22(12), 1250301 (2012).

							
							\bibitem{brolin} Brolin, H. Invariant sets under iteration of rational functions. Ark. Mat., 6(2), 103-144 (1965).
							
							\bibitem{CG} Carleson, L. and Gamelin, T. W., Complex dynamics, Springer-Verlag, New York 1993. 
							
							\bibitem{juliasetortpoly} Christiansen, J. S., Henriksen, C., Pedersen, H. L.,  Petersen, C. L. Julia sets of orthogonal polynomials. Potential Anal., 50(3), 401-413 (2019).
							
							\bibitem{filled julia set of cheb. polyn} Christiansen, J. S., Henriksen, C., Pedersen, H. L., Petersen, C. L. Filled Julia sets of Chebyshev polynomials. J. Geom. Anal., 1-14 (2021).
							
						\bibitem{duncan} Dauvergne, D. (2019). A necessary and sufficient condition for global convergence of the zeros of random polynomials. Adv. Math., (384) Paper No. 107691, 33 (2021)
										
							\bibitem{KRS} Kaschner, S. R., Romero, R., Simmons, D. Geometric Limits of Julia Sets of Maps $z^n+\exp(2\pi i\theta)$ as $n\to\infty$. International Journal of Bifurcation and Chaos, 25(08), 1530021 (2015).
							
							\bibitem{Klimek} Klimek, M. Metrics associated with extremal plurisubharmonic functions. Proc. Amer. Math. Soc., 123(9), 2763-2770 (1995).
							
							\bibitem{zeros of faber poly} Kövari, T.,  Pommerenke, C. On Faber polynomials and Faber expansions. Math. Z., 99(3), 193-206 (1967).
							
							\bibitem{defin faber poly} Levenberg, N.,  Wielonsky, F.  Zeros of Faber polynomials for Joukowski airfoils. Constr. Approx., 52(1), 93-114 (2020).
							
							\bibitem{L} Lindsey, K. A.  Shapes of polynomial Julia sets. Ergodic Theory Dynam. Systems, 35(6), 1913-1924 (2015).
							
							\bibitem{LY} Lindsey, K.,  Younsi, M.  Fekete polynomials and shapes of Julia sets. Trans. Amer. Math. Soc., 371(12), 8489-8511 (2019)..
							
							\bibitem{lyubich} Lyubich, M. Y. The maximum-entropy measure of a rational endomorphism of the Riemann sphere. Functional Analysis and Its Applications, 16(4), 309-311 (1982).
							
							\bibitem{weaklmtof maximal entropy} Petersen, C. L.,  Uhre, E. Weak limits of the measures of maximal entropy for Orthogonal polynomials. Potential Anal., 54(2), 219-225 (2021).
							
							\bibitem{ransfordbook} Ransford, T.  Potential theory in the complex plane (No. 28). Cambridge university press (1995).
							
							\bibitem{SaTo} Saff, E. B.,  Totik, V.  Logarithmic potentials with external fields (Vol. 316). Springer Science  Business Media (2013).
							
							\bibitem{Sibony}Sibony, N.  Dynamique des applications rationnelles de $P^k$. Panoramas et syntheses, 8, 97-185 (1999).
							
							\bibitem{extremal plrsubhrmnic fnction} Siciak, J. Extremal plurisubharmonic functions in $ C^ N$. Ann. Polon. Math., 39(1), 175-211 (1981).
							
							\bibitem{Siciak2} Siciak, J.  On metrics associated with extremal plurisubharmonic functions. Bull. Polish Acad. Sci. Math., no. 2, 151-161 (1997).
							
							\bibitem{general orth. polyn.} Stahl, H., Steel, J., Totik, V. General orthogonal polynomials (No. 43). Cambridge University Press (1992). 							
						\end{thebibliography}
					\end{document}